%% file: nips_2017_Sugiyama2.tex
\documentclass{article}

\PassOptionsToPackage{numbers, compress}{natbib}
\usepackage[utf8]{inputenc} 
\usepackage[T1]{fontenc}    
\usepackage{hyperref}       
\usepackage{url}            
\usepackage{booktabs}       
\usepackage{amsfonts}       
\usepackage{microtype}      

\usepackage{amsthm}
\usepackage{amssymb}
\usepackage[cmex10]{amsmath}
\usepackage[pdftex]{graphicx}
\graphicspath{{./figures/}{./biography/} }
\usepackage{algpseudocode}
\usepackage{algorithm}
\usepackage{multirow}
\usepackage{pbox}
\usepackage{textcomp}
\usepackage{authblk}
\usepackage[ margin=1in]{geometry}

\input{DefinedComand}
\newtheorem{theorem}{Theorem}

\title{Learning Efficient Tensor Representations\\
with Ring Structure Networks
}

\author[1]{Qibin Zhao}
\author[1]{Masashi Sugiyama}
\author[2]{Andrzej Cichocki}
\affil[1]{RIKEN AIP, Tokyo, Japan}
\affil[2]{RIKEN BSI, Saitama, Japan}

\date{}
\begin{document}

\maketitle

\begin{abstract}
\emph{Tensor train (TT) decomposition} is a powerful representation for high-order tensors,
which has been successfully applied to various machine learning tasks in recent years. However,
since the tensor product is not commutative,
permutation of data dimensions makes
solutions and TT-ranks of TT decomposition inconsistent.
To alleviate this problem, we propose a permutation symmetric network structure by
employing circular multilinear products over a sequence of low-order core tensors.
This network structure can be graphically interpreted as a cyclic interconnection of tensors, and thus
we call it \emph{tensor ring (TR) representation}. We develop several efficient algorithms
to learn TR representation with adaptive TR-ranks by employing low-rank
approximations. Furthermore, mathematical properties are investigated, which
enables us to perform basic operations in a computationally efficiently way by using TR representations.
Experimental results on synthetic signals and real-world datasets demonstrate that
the proposed TR network is more expressive and consistently informative than existing TT networks.
\end{abstract}

\section{Introduction}
\emph{Tensor decompositions} aim to represent a higher-order (or multi-dimensional) data as a multilinear product of several latent factors, which attracted considerable attentions in machine learning \cite{anandkumar2014tensor,romera2013multilinear,kanagawa2016gaussian} and signal processing~\cite{cong2015tensor,Zhou-PIEEE} in recent years.
For a $d$th-order ``square'' tensor of size $n$ with ``square'' core tensor of size $r$,
standard tensor decompositions are the \emph{canonical polyadic (CP) decomposition} \cite{bro1997parafac,CPD-Comon15,zhao2015bayesian} which represents data as a sum of rank-one tensors by $\mathcal{O}(dnr)$ parameters
and \emph{Tucker decomposition}~\cite{tucker1966some,HOOI:Lathauwer:2000,qiinfinite,wu2014multifactor} which represents data as  a core tensor and several factor matrices by $\mathcal{O}(dnr + r^d)$ parameters.
 In general, CP decomposition provides a compact representation but with difficulties in finding the optimal solution, while Tucker decomposition is stable and flexible but its number of parameters scales exponentially to the tensor order. 

Recently, \emph{tensor networks} have emerged as a powerful tool for analyzing very high-order tensors~\cite{cichocki2016tensor}. A powerful tensor network is \emph{tensor train / matrix product states} (TT/MPS) representation~\cite{oseledets2011tensor}, which requires $\mathcal{O}(dnr^2)$ parameters and avoid the curse of dimensionality through a particular geometry of low-order contracted tensors. TT representation has been applied to model weight parameters in deep neural network and nonlinear kernel learning~\cite{novikov2015tensorizing,NIPS2016_6211}, achieving a significant compression factor and scalability. It also has been successfully used for feature learning and classification~\cite{7207289}. It was shown in~\cite{oseledets2010tt} that TT decomposition with minimal possible compression ranks always exists and can be computed by a sequence of singular value decompositions (SVDs), or by the cross approximation algorithm.  

Although TT decomposition has gained a success in tackling various machine learning tasks, there are some major limitations including that i) the constraint on TT-ranks, i.e., $r_1=r_{d+1}=1$, leads to the limited representation ability and flexibility; ii) TT-ranks are  small in the border cores and large in the middle cores, which might not be optimal for a given data tensor;  iii)  the permutation of data tensor will yield an inconsistent solution, i.e., TT representations and TT-ranks are sensitive to the order of tensor dimensions. Hence, finding the optimal permutation remains a challenging problem.

By taking into account these limitations of TT decomposition, we introduce a new structure of tensor networks, which can be considered as a generalization of TT representations. First of all, we relax the condition over TT-ranks, i.e., $r_1=r_{d+1}=1$, leading to an enhanced representation ability. Secondly, the strict ordering of multilinear products between cores should be alleviated. Third, the cores should be  treated equivalently by making the model symmetric. To this end, we add a new connection between the first  and the last core tensors, yielding a circular tensor products of a set of cores. More specifically, we consider that each tensor element is approximated by performing a trace operation over the sequential multilinear products of cores. Since the trace operation ensures a scalar output, $r_1=r_{d+1}=1$ is not necessary. In addition, the cores can be circularly shifted and treated equivalently due to the properties of the trace operation. By using the graphical illustration (see Fig.~\ref{fig:TRD}), this concept implies that the cores are interconnected circularly, which looks like a ring structure. Hence, we call this model \emph{tensor ring (TR) decomposition} and its cores \emph{tensor ring (TR) representations}.  

\begin{figure}[t]
  \centering
  \includegraphics[width=0.6\columnwidth]{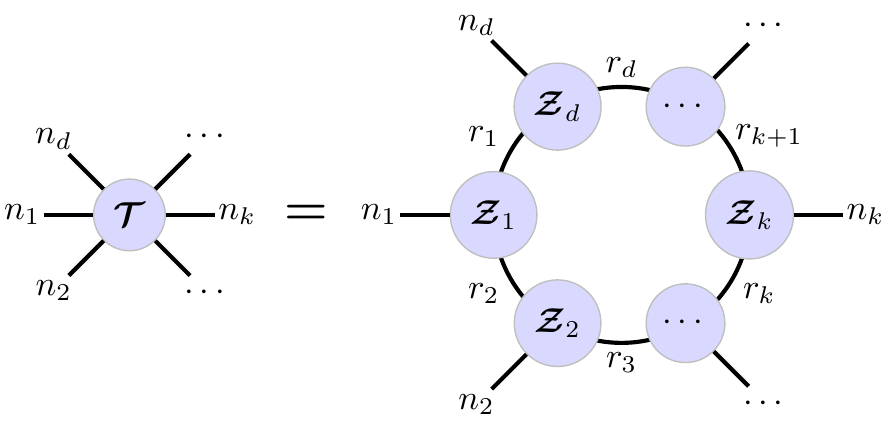}
  \caption{A graphical representation of tensor ring decomposition.}
  \label{fig:TRD}
\end{figure}

To learn TR representations, we firstly develop a non-iterative TR-SVD algorithm that is computationally efficient and scalable. To obtain a low-rank TT representation, we also develop a block-wise alternating least-squares (ALS) algorithm, which updates the tensor products of two adjacent cores first; then a low-rank approximation is employed to separate this term into two cores with the lowest rank.
We experimentally demonstrate the usefulness of the proposed approach on both synthetic and real-world datasets.


\section{Tensor Ring Decomposition}
\label{sec:trm}
The TR decomposition aims to represent a high-order (or multi-dimensional) tensor by a sequence of 3rd-order tensors that are multiplied circularly. Specifically, let $\tensor{T}$ be a $d$th-order  tensor of size $n_1\times n_2\times \cdots\times n_d$, denoted by $\tensor T\in\mathbb{R}^{n_1\times \cdots\times n_d}$, TR representation is to decompose it into a sequence of latent tensors $\tensor Z_k\in\mathbb{R}^{r_k\times n_k\times r_{k+1}}, k=1,2,\ldots, d$, which can be expressed in an element-wise form given by
\begin{equation}
\label{eq:TRD1}
\begin{split}
T(i_1,i_2,\ldots,i_d) =\text{Tr}\left\{\mat Z_1(i_1)\mat Z_2(i_2)\cdots \mat Z_d(i_d)\right\} 
= \text{Tr}\left\{\prod_{k=1}^d \mat Z_k(i_k)\right\}.
\end{split}
\end{equation}
$T(i_1,i_2,\ldots,i_d)$ denotes the $(i_1,i_2,\ldots,i_d)$th element of the tensor. $\mat Z_k(i_k)$ denotes the $i_k$th lateral slice matrix of the latent tensor $\tensor Z_k$, which is of size $r_k\times r_{k+1}$. Note that any two adjacent latent tensors, $\tensor Z_k$ and $\tensor Z_{k+1}$, have a common dimension $r_{k+1}$ on their corresponding modes. The last latent tensor $\tensor Z_d$ is of size $r_d\times n_d\times r_1$, i.e., $r_{d+1}=r_1$, which ensures the product of these matrices is a square matrix. These prerequisites play  key roles in TR decomposition, resulting in some important numerical properties.  For simplicity,  the latent tensor $\tensor Z_k$ can also be called the $k$th-\emph{core} (or \emph{node}). The size of cores, $r_k, k=1,2,\ldots, d$,  collected and denoted by a vector $\vect r = [r_1, r_2,\ldots, r_d]^T$, are called \emph{TR-ranks}. From (\ref{eq:TRD1}), we can observe that $T(i_1,i_2,\ldots,i_d)$ is equivalent to the trace of a sequential product of matrices $\{\mat Z_k(i_k)\}$.  Based on (\ref{eq:TRD1}), we can also express TR decomposition in the tensor form, given by
\begin{equation}
\tensor T = \sum_{\alpha_1,\ldots,\alpha_d=1}^{r_1,\ldots,r_d}\mat z_1(\alpha_1,\alpha_2)\circ \mat z_2(\alpha_2,\alpha_3) \circ \cdots \circ \mat z_d(\alpha_{d},\alpha_1),
\nonumber
\end{equation}
where the symbol `$\circ$' denotes the outer product of vectors and $\mat z_k(\alpha_k,\alpha_{k+1})\in\mathbb{R}^{n_k}$ denotes the ($\alpha_k,\alpha_{k+1}$)th mode-2 fiber of tensor $\tensor Z_k$. The number of parameters in TR representation is $\mathcal{O}(dnr^2)$, which is linear to the tensor order $d$.

The TR representation can also be illustrated graphically by a linear tensor network as shown in Fig.~\ref{fig:TRD}. A node represents a tensor (including a matrix and a vector) whose order is denoted by the number of edges. The number by an edge specifies the size of each mode (or dimension). The connection between two nodes denotes a multilinear product operator between two tensors on a specific mode. This is also called \emph{tensor contraction}, which corresponds to the summation over the indices of that mode. It should be noted that $\tensor Z_d$ is connected to $\tensor Z_1$ by the summation over the index $\alpha_1$, which corresponds to the trace operation.   For simplicity, we denote TR decomposition by $\tensor T = \Re(\tensor Z_1, \tensor Z_2, \ldots, \tensor Z_d)$.

\begin{theorem}[Circular dimensional permutation invariance]
\label{theorem:invariance}
Let  $\tensor T\in\mathbb{R}^{n_1\times n_2\times \ldots\times n_d}$ be a $d$th-order tensor and its TR decomposition is given by  $\tensor T = \Re(\tensor Z_1, \tensor Z_2, \ldots, \tensor Z_d)$. If we define ${\overleftarrow {\tensor {T}}^k}\in \mathbb{R}^{n_{k+1}\times \cdots\times n_d\times n_1\times\cdots\times n_{k}}$ as the circularly shifted version along  the dimensions of $\tensor T$ by k, then we have ${\overleftarrow {\tensor {T}}^k} =\Re(\tensor Z_{k+1}, \ldots, \tensor Z_d, \tensor Z_{1},\ldots \tensor Z_k)$.
\end{theorem}
A proof of Theorem~\ref{theorem:invariance} is provided in Appendix~\ref{sec:theorem:invariance}.

It should be noted that circular dimensional permutation invariance is an essential feature that distinguishes TR decomposition from TT decomposition. For TT decomposition, the product of matrices must keep a strictly sequential order, yielding that the cores for representing the same tensor with a circular dimension shifting cannot keep invariance. Hence, it is necessary to choose an optimal dimensional permutation when applying the TT decomposition.

\section{Sequential SVDs Algorithm}
\label{sec:TRSVD}

We propose the first algorithm for computing the TR decomposition using $d$ sequential SVDs. This algorithm will be called the \emph{TR-SVD algorithm}.

\begin{theorem}\label{theorem:TR}
Let us assume $\tensor T$ can be represented by a TR decomposition. If  the $k$-unfolding matrix $\mat T_{\langle k\rangle}$ has $Rank(\mat T_{\langle k \rangle})=R_{k+1}$, then there exists a TR decomposition with TR-ranks $\vect r$ which satisfies that $ \exists k, r_1r_{k+1} \leq R_{k+1}$.
\end{theorem}
\begin{proof}
We can express TR decomposition in the form of $k$-unfolding matrix,
\begin{equation}
\label{eq:TRSVD1}
\begin{split}
 T_{\langle k\rangle}(\overline{i_1\cdots i_k}, \overline{i_{k+1}\cdots i_d})  
 = \text{Tr}\!\left\{\! \prod_{j=1}^k \mat Z_j(i_j)\!\!\!\!  \prod_{j=k+1}^d \mat Z_j(i_j)\! \right\}\!
 = \left\langle\!\!\text{vec}\!\!\left(\prod_{j=1}^k \mat Z_j(i_j)\right)\!\!, \text{vec}\!\!\left(\prod_{j=d}^{k+1} \mat Z^T_j(i_j)\!\! \right)\!\!   \right\rangle.
\end{split}
\end{equation}
It can also be rewritten as
\begin{equation}
\label{eq:TRSVD2}
\begin{split}
 T_{\langle k\rangle}(\overline{i_1\cdots i_k}, \overline{i_{k+1}\cdots i_d}) 
 = \sum_{\alpha_1\alpha_{k+1}} Z^{\leq k}\left(\overline{i_1\cdots i_k}, \overline{\alpha_1\alpha_{k+1}}\right)  Z^{>k}\left(\overline{\alpha_1\alpha_{k+1}},\overline{i_{k+1}\cdots i_d}\right),
\end{split}
\end{equation}
where we defined the subchain by merging multiple linked cores as $\mat Z^{<k}(\overline{i_1 \cdots i_{k-1}}) = \prod_{j=1}^{k-1} \mat Z_j(i_j)$ and $\mat Z^{>k}(\overline{i_{k+1} \cdots i_{d}}) = \prod_{j=k+1}^{d} \mat Z_j(i_j)$. Hence, we can obtain $\mat T_{\langle k\rangle} = \mat Z^{\leq k}_{(2)} (\mat Z^{>k}_{[2]})^T$, where the subchain $\mat Z^{\leq k}_{(2)}$ is of size $\prod_{j=1}^k n_j\times r_1r_{k+1}$, and $\mat Z^{>k}_{[2]}$ is of size $\prod_{j=k+1}^d n_{j}\times r_1r_{k+1}$. Since the rank of $\mat T_{\langle k\rangle}$ is $R_{k+1}$,  we can obtain $r_1r_{k+1} \leq R_{k+1}$.
\end{proof}

According to (\ref{eq:TRSVD1}) and (\ref{eq:TRSVD2}), TR decomposition can be written as
\begin{equation}
T_{\langle 1\rangle}(i_1, \overline{i_2\cdots i_d}) = \sum_{\alpha_1,\alpha_2} Z^{\leq 1}(i_1, \overline{\alpha_1\alpha_2})Z^{>1}(\overline{\alpha_1\alpha_2}, \overline{i_2\cdots i_d}).
\nonumber
\end{equation}
Since the low-rank approximation of $\mat T_{\langle 1\rangle}$ can be obtained by the truncated SVD, which is
\begin{equation}
\mat T_{\langle 1 \rangle} = \mat U \Sigma\mat V^T +\mat E_{1},
\nonumber
\end{equation}
the first core $\tensor Z_1 (i.e., \tensor Z^{\leq 1}) $ of size $r_1\times n_1\times r_2$ can be obtained by the proper reshaping and permutation of $\mat U$ and the subchain $\tensor Z^{>1}$ of size $r_2\times \prod_{j=2}^d n_j\times r_1$ is obtained by the proper reshaping and permutation of $\mat \Sigma \mat V^T$, which corresponds to the remaining $d-1$ dimensions of $\tensor T$.  Subsequently, we can further reshape the subchain $\tensor Z^{>1}$ as a matrix $\mat Z^{>1}\in\mathbb{R}^{r_2n_2\times \prod_{j=3}^d n_j r_1}$ which thus can be written as
\begin{equation}
Z^{>1}(\overline{\alpha_2 i_2}, \overline{i_3\cdots i_d\alpha_1}) = \sum_{\alpha_3}Z_2(\overline{\alpha_2 i_2}, \alpha_3) Z^{>2}(\alpha_3, \overline{i_3\cdots i_d\alpha_1}).
\nonumber
\end{equation}
By applying truncated SVD, i.e., $\mat Z^{>1} = \mat U\Sigma\mat V^T + \mat E_2$, we can obtain the second core $\tensor Z_2$ of size $(r_2\times n_2\times r_3)$ by appropriately reshaping $\mat U$ and the subchain $\tensor Z^{>2}$ by proper reshaping of $\mat \Sigma \mat V^T$. This procedure can be performed sequentially to obtain all $d$ cores $\tensor Z_k, k=1,\ldots, d$.

As proved in \cite{oseledets2011tensor}, the approximation error by using such sequential SVDs is given by
\begin{equation}
\|\tensor{T}-\Re(\tensor Z_1, \tensor Z_2, \ldots, \tensor Z_d)\|_F \leq \sqrt{\sum_{k=1}^{d-1} \|\mat E_k\|_F^2}.
\nonumber
\end{equation}
Hence, given a prescribed relative error $\epsilon_p$, the truncation threshold $\delta$ can be set to $\frac{\epsilon_p}{\sqrt{d-1}}\|\tensor T\|_F$. However, considering that $\|\mat E_1\|_F$ corresponds to two ranks including both $r_1$ and $r_2$, while $\mat \|\mat E_k\|_F, \forall k>1$ correspond to only one rank $r_{k+1}$. Therefore, we modify the truncation threshold as
\begin{equation}
\label{eq:TRSVDthreshold}
\delta_k =\left \{\begin{array}{ll}
             \sqrt{2}\epsilon_p \|\tensor T\|_F / \sqrt{d} & k=1, \\
             \epsilon_p\|\tensor T\|_F / \sqrt{d} & k>1. \\
           \end{array} \right.
\end{equation}
A pseudocode of the TR-SVD algorithm is summarized in Alg.~\ref{alg:TR-SVD}.

\renewcommand{\algorithmicrequire}{\textbf{Input:}}
\renewcommand{\algorithmicensure}{\textbf{Output:}}
\begin{algorithm}[t]
\caption{TR-SVD}
\label{alg:TR-SVD}
\begin{algorithmic}[1]
\Require A $d$th-order tensor $\tensor T$ of size $(n_1\times\cdots\times n_d)$ and the prescribed relative error $\epsilon_p$.
\Ensure  Cores $\tensor Z_k, k=1,\ldots,d$ of TR decomposition and the TR-ranks $\mat r$.
\State Compute truncation threshold $\delta_k$ for $k=1$ and $k>1$.
\State Choose one mode as the start point (e.g., the first mode) and obtain the $1$-unfolding matrix $\mat T_{\langle 1\rangle}$.
\State Low-rank approximation by applying $\delta_1$-truncated SVD: $\mat T_{\langle 1 \rangle} = \mat U\Sigma\mat V^T + \mat E_1$.
\State Split ranks $r_1,r_2$ by
$\min_{r_1,r_2} \quad \|r_1-r_2\|, \quad
 \text{s. t.} \quad r_1r_2 = \text{rank}_{\delta_1}(\mat T_{\langle 1\rangle})$.
\State $\tensor Z_1 \gets \text{permute}(\text{reshape}(\mat U, [ n_1, r_1, r_2]), [2,1,3])$.
\State $\tensor Z^{>1} \gets \text{permute}(\text{reshape}(\mat \Sigma\mat V^T, [r_1,r_2,\prod_{j=2}^d n_j]), [2,3,1])$.
\For{ $k=2$ to $d-1$}
\State $\mat Z^{>k-1} = \text{reshape}(\tensor Z^{>k-1}, [r_k n_k, n_{k+1}\cdots n_d r_1])$.
\State Compute $\delta_k$-truncated SVD:
             $\mat Z^{>k-1} = \mat U\Sigma\mat V^T + \mat E_k$.
\State $r_{k+1} \gets \text{rank}_{\delta_k}(\mat Z^{>k-1})$.
\State $\tensor Z_k \gets \text{reshape}(\mat U, [r_k, n_k, r_{k+1}])$.
\State $\tensor Z^{>k}\gets \text{reshape}(\Sigma\mat V^T, [r_{k+1}, \prod_{j=k+1}^{d} n_j, r_1])$.
\EndFor
\end{algorithmic}
\end{algorithm}

The cores obtained by the TR-SVD algorithm are left-orthogonal, which is $\mat Z^T_{k\langle 2\rangle}\mat Z_{k\langle 2\rangle}=\mat I$ for $k=2,\ldots, d-1$. It should be noted that TR-SVD is a non-recursive algorithm that does not need iterations for convergence. However, it might obtain different representations by choosing a different mode as the start point. This indicates that TR-ranks $\mat r$ is not necessary to be the global optimum in TR-SVD. 

\section{Block-Wise Alternating Least-Squares (ALS) Algorithm}
\label{sec:BALS}

The ALS algorithm has been widely applied to various tensor decomposition models such as CP and Tucker decompositions~\cite{Kolda09,Holtz-TT-2012}. The main concept of ALS is optimizing one core while the other cores are fixed, and this procedure will be repeated until some convergence criterion is satisfied. Given a $d$th-order tensor $\tensor T$, our goal is optimize the error function as
\begin{equation}
\label{eq:alsobj}
\min_{\tensor Z_1,\ldots, \tensor Z_d }\|\tensor T - \Re(\tensor Z_1, \ldots, \tensor Z_d)\|_F.
\end{equation}

According to the TR definition in (\ref{eq:TRD1}), we have
\begin{equation}
\begin{split}
&
T(i_1,i_2,\ldots,i_d)  = \sum_{\alpha_1,\ldots,\alpha_d}Z_1(\alpha_1,i_1,\alpha_2)Z_2(\alpha_2,i_2,\alpha_3) \cdots Z_d(\alpha_{d},i_d,\alpha_1)\\
&
 =\sum_{\alpha_{k},\alpha_{k+1}} \Big\{Z_k(\alpha_k,i_k, \alpha_{k+1}) Z^{\neq k} (\alpha_{k+1}, \overline{i_{k+1}\cdots i_d i_1 \cdots i_{k-1}} ,\alpha_{k})\Big\},
\end{split}
\nonumber
\end{equation}
where $\mat Z^{\neq k}(\overline{i_{k+1} \cdots i_{d}i_1\ldots i_{k-1} }) = \prod_{j=k+1}^{d} \mat Z_j(i_j)\prod_{j=1}^{k-1} \mat Z_j(i_j)$ denotes a slice matrix of subchain tensor by  merging all cores except $k$th core $\tensor Z_k$. Hence, the mode-$k$ unfolding matrix of $\tensor T$ can be expressed by
\begin{equation}
\begin{split}
 T_{[k]}(i_k, \overline{i_{k+1} \cdots i_d i_1 \cdots i_{k-1}}) = \sum_{\alpha_k\alpha_{k+1}}\Big\{ Z_k(i_k, \overline{\alpha_k \alpha_{k+1}}) 
 Z^{\neq k}(\overline{\alpha_k\alpha_{k+1}},\overline{i_{k+1}\cdots i_d i_1\cdots i_{k-1}})\Big\}.
\end{split}
\nonumber
\end{equation}
By applying different mode-$k$ unfolding operations, we can obtain that
$\mat T_{[k]} = {\mat Z_k}_{(2)} \left(\mat Z^{\neq k}_{[2]}\right)^T $, where $\tensor Z^{\neq k}$  is a subchain obtained by merging $d-1$ cores.

The objective function in (\ref{eq:alsobj}) can be optimized by solving $d$ subproblems alternatively. More specifically,  having fixed all but one core, the problem reduces to a linear least squares problem, which is
\begin{equation}
\min_{\mat Z_{k(2)}} \Big\| \mat T_{[k]} - \mat Z_{k(2)} \left(\mat Z^{\neq k}_{[2]}\right)^T \big\|_F, \quad k=1,\ldots,d.
\nonumber
\end{equation}

Here, we propose a computationally efficient block-wise ALS (BALS) algorithm by utilizing truncated SVD, which facilitates the self-adaptation of ranks. The main idea is to perform the blockwise optimization followed by the separation of a block into individual cores. To achieve this, we consider merging two linked  cores, e.g., $\tensor Z_k, \tensor Z_{k+1}$, into a block (or subchain) $\tensor Z^{(k,k+1)}\in \mathbb{R}^{r_k\times n_k n_{k+1}\times r_{k+2}}$.  Thus, the subchain $\tensor Z^{(k,k+1)}$ can be optimized while leaving all cores except $\tensor Z_k, \tensor Z_{k+1}$ fixed. Subsequently, the subchain $\tensor Z^{(k,k+1)}$ can be reshaped into $\tilde{\mat Z}^{(k,k+1)}\in\mathbb{R}^{r_kn_k\times n_{k+1}r_{k+2}}$ and separated into a left-orthonormal core $\tensor Z_k$ and $\tensor Z_{k+1}$ by a truncated SVD:
\begin{equation}
\label{eq:BALS1}
\tilde{\mat Z}^{(k,k+1)} = \mat U \mat \Sigma \mat V^T = \mat Z_{k\langle 2\rangle }\mat Z_{k+1\langle 1\rangle},
\end{equation}
where $\mat Z_{k\langle 2\rangle }\in\mathbb{R}^{r_kn_k\times r_{k+1}} $ is the $2$-unfolding matrix of core $\tensor Z_k$, which can be set to $\mat U$, while $\mat Z_{k+1\langle 1\rangle}\in \mathbb{R}^{r_{k+1}\times n_{k+1}r_{k+2}}$ is the $1$-unfolding matrix of core $\tensor Z_{k+1}$, which can be set to $\mat \Sigma\mat V^T $. This procedure thus moves on to optimize the next block cores $\tensor Z^{(k+1,k+2)}, \ldots, \tensor Z^{(d-1,d)},\tensor Z^{(d,1)}$ successively in the similar way. Note that since the TR model is circular, the $d$th core  can also be merged with the first core yielding the block core $\tensor Z^{(d,1)}$.

The key advantage of our BALS algorithm is the rank adaptation ability which can be achieved simply by separating the block core into two cores via truncated SVD, as shown in (\ref{eq:BALS1}). The truncated rank $r_{k+1}$ can be chosen such that the approximation error is below a certain threshold. One possible choice is to use the same threshold as in the TR-SVD algorithm, i.e., $\delta_k$ described in (\ref{eq:TRSVDthreshold}). However, the empirical experience shows that this threshold  often leads to overfitting and the truncated rank is higher than the optimal rank. This is because the updated block $\tensor Z^{(k,k+1)}$ during ALS iterations is not a closed form solution and many iterations are necessary for convergence. To relieve this problem, we choose the truncation threshold based on both the current and the desired approximation errors, which is
\begin{equation}
\nonumber
\delta = \max \left\{\epsilon\|\tensor{T}\|_F/\sqrt{d},\, \epsilon_{p}\|\tensor{T}\|_F/ \sqrt{d}\right\}.
\end{equation}
A pseudo code of the BALS algorithm is described in Alg.~\ref{alg:TR-BALS}.

\begin{algorithm}[t]
\caption{TR-BALS}
\label{alg:TR-BALS}
\begin{algorithmic}[1]
\Require A $d$-dimensional tensor $\tensor T$ of size $(n_1\times\cdots\times n_d)$ and the prescribed relative error $\epsilon_{p}$.
\Ensure  Cores $\tensor Z_k$ and TR-ranks $r_k$, $k=1,\ldots,d$.
\State Initialize $r_k =1$ for $k=1,\ldots, d$.
\State Initialize $\tensor Z_k\in\mathbb{R}^{r_{k}\times n_k\times r_{k+1}}$ for $k=1,\ldots, d$.
\Repeat $\quad k\in\text{circular}\{1,2,\ldots, d\}$;
\State Compute the subchain $\tensor Z^{\neq (k,k+1)}$. 
\State \parbox[t]{\dimexpr\linewidth-\algorithmicindent-0.1in} {Obtain the mode-2 unfolding matrix $\mat Z^{\neq (k,k+1)}_{[2]}$ of size $\prod_{j=1}^d n_j/(n_k n_{k+1})\times r_k r_{k+2}$. \strut}
\State $\mat Z^{(k,k+1)}_{(2)} \gets \arg\min \left\|{\mat T_{[k]} - \mat Z^{(k,k+1)}_{(2)} \left(\mat Z_{[2]}^{\neq (k, k+1)}\right)^T}\right\|_F$.
\State Tensorization of mode-2 unfolding matrix
\begin{equation}\nonumber \tensor Z^{(k,k+1)} \gets \text{folding}(\mat Z^{(k,k+1)}_{(2)}).\end{equation}
\State Reshape the block core by
\begin{equation}\nonumber \tilde{\mat Z}^{(k,k+1)} \gets \text{reshape}(\tensor Z^{(k,k+1)}, [r_k n_k\times n_{k+1}r_{k+2} ]).   \end{equation}
\State \parbox[t]{\dimexpr\linewidth-\algorithmicindent-0.1in} {Low-rank approximation by $\delta$-truncated SVD
         $\tilde{\mat Z}^{(k,k+1)} = \mat U \mat \Sigma \mat V^T$. \strut}
\State $\tensor Z_k \gets \text{reshape}(\mat U, [r_k,n_k,r_{k+1}])$.
\State $\tensor Z_{k+1} \gets \text{reshape}(\Sigma \mat V^T, [r_{k+1},n_{k+1},r_{k+2}])$.
\State $r_{k+1} \gets \text{rank}_{\delta} (\tilde{\mat Z}^{(k,k+1)})$.
\State $k\gets k+1$.
\Until{The desired approximation accuracy is achieved, i.e., $\epsilon \leq \epsilon_p$. }
\end{algorithmic}
\end{algorithm}

\section{Properties of TR Representation}
\label{sec:property}
 By assuming that tensor data have been already represented as TR decompositions, i.e., a sequence of third-order cores, we justify and demonstrate that the basic operations on tensors, such as the \emph{addition}, \emph{multilinear product}, \emph{Hadamard product}, \emph{inner product} and \emph{Frobenius norm}, can be performed efficiently by the appropriate operations on each individual cores.
We have the following theorems:

\begin{theorem}\label{theorem:T1+T2}
Let $\tensor T_1$ and $\tensor T_2$ be $d$th-order tensors of size $n_1\times \cdots\times n_d$. If TR decompositions of these two tensors are $\tensor T_1 = \Re(\tensor Z_1,\ldots,\tensor Z_d)$ where $\tensor Z_k\in\mathbb{R}^{r_k\times n_k\times r_{k+1}}$ and $\tensor T_2 = \Re(\tensor Y_1,\ldots,\tensor Y_d)$ where $\tensor Y_k\in\mathbb{R}^{s_k\times n_k\times s_{k+1}}$, then the addition of these two tensors, $\tensor T_3= \tensor T_1 + \tensor T_2$, can also be represented in the TR format given by $\tensor T_3 = \Re(\tensor X_1, \ldots, \tensor X_d)$, where $\tensor X_k\in\mathbb{R}^{q_k\times n_k\times q_{k+1}}$ and $q_k = r_k+s_k$. Each core $\tensor X_k$ can be computed by
\begin{equation}
\label{eq:TRsum1}
\mat X_k(i_k) = \left(
                  \begin{array}{cc}
                    \mat Z_k(i_k) & 0 \\
                    0 & \mat Y_k(i_k) \\
                  \end{array}
                \right),
                  \begin{array}{c}
                     i_k=1,\ldots,n_k,\\
                    k=1,\ldots, d.\\
                  \end{array}
\end{equation}
\end{theorem}
A proof of Theorem~\ref{theorem:T1+T2} is provided in Appendix~\ref{sec:theorem:T1+T2}.
Note that the sizes of new cores are increased and not optimal in general. This problem can be solved by the rounding procedure \cite{oseledets2011tensor}.

\begin{theorem}
\label{theorem:TRtimesvectors}
Let $\tensor T\in\mathbb{R}^{n_1\times\cdots\times n_d}$ be a $d$th-order tensor whose TR representation is $\tensor T = \Re(\tensor Z_1,\ldots,\tensor Z_d)$ and $\mat u_k\in\mathbb{R}^{n_k}, k=1,\ldots,d$ be a set of vectors, then the multilinear products, denoted by $c=\tensor T \times_1 \mat u_1^T\times_2\cdots \times_d \mat u_d^T$, can be computed by the multilinear product on each cores, which is
\begin{equation}
\label{eq:TRmultiproduct}
\begin{split}
c=\Re(\mat X_1, \ldots, \mat X_d) \; \text{where} \; \mat X_k = \sum_{i_k=1}^{n_k}\mat Z_k(i_k) u_k(i_k).
\end{split}
\end{equation}
\end{theorem}
A proof of Theorem~\ref{theorem:TRtimesvectors} is provided in Appendix~\ref{sec:theorem:TRtimesvectors}.
It should be noted that the computational complexity in the original tensor form is $\mathcal{O}(dn^d)$, while it reduces to $\mathcal{O}(dnr^2 + dr^3)$ that is linear to tensor order $d$ by using TR representation.

\begin{theorem}\label{theorem:T1HadamardT2}
Let $\tensor T_1$ and $\tensor T_2$ be $d$th-order tensors of size $n_1\times \cdots\times n_d$. If the TR decompositions of these two tensors are $\tensor T_1 = \Re(\tensor Z_1,\ldots,\tensor Z_d)$ where $\tensor Z_k\in\mathbb{R}^{r_k\times n_k\times r_{k+1}}$ and $\tensor T_2 = \Re(\tensor Y_1,\ldots,\tensor Y_d)$ where $\tensor Y_k\in\mathbb{R}^{s_k\times n_k\times s_{k+1}}$, then the Hadamard product of these two tensors, $\tensor T_3= \tensor T_1 \circledast \tensor T_2$, can also be represented in the TR format given by $\tensor T_3 = \Re(\tensor X_1, \ldots, \tensor X_d)$, where $\tensor X_k\in\mathbb{R}^{q_k\times n_k\times q_{k+1}}$ and $q_k = r_ks_k$. Each core $\tensor X_k$ can be computed by
\begin{equation}
\label{eq:TRhadama}
\mat X_k(i_k) = \mat Z_k(i_k) \otimes \mat Y_k(i_k), \quad k=1,\ldots,d.
\end{equation}
\end{theorem}
A proof of Theorem~\ref{theorem:T1HadamardT2} is provided in Appendix~\ref{sec:theorem:T1HadamardT2}.
Furthermore, one can compute the \emph{inner product} of two tensors in TR representations. For two tensors $\tensor T_1$ and $\tensor T_2$, it is defined as
$
\langle \tensor T_1, \tensor T_2\rangle  =  \sum_{i_1,\ldots,i_d} T_3(i_1,\ldots,i_d),
$
where $\tensor T_3 =  \tensor T_1 \circledast \tensor T_2$. Thus, the inner product can be computed by applying the Hadamard product and then computing the multilinear product between $\tensor T_3$ and vectors of all ones, i.e., $\mat u_k=\mat 1, k=1,\ldots,d$.  In contrast to $\mathcal{O}(n^d)$ in the original tensor form, the computational complexity is equal to $\mathcal{O}(dnq^2+ dq^3)$ that is linear to $d$ by using TR representation. Similarly, we can also compute the \emph{Frobenius norm} $\|\tensor T \|_F = \sqrt{\langle\tensor T, \tensor T \rangle}$ in the TR representation.

\section{Experimental Results}
\label{sec:experiment}
In this section, we experimentally demonstrate the usefulness of the proposed approach.

\subsection{Numerical Illustration}
\begin{figure}
  \centering
  \includegraphics[width=0.9\columnwidth]{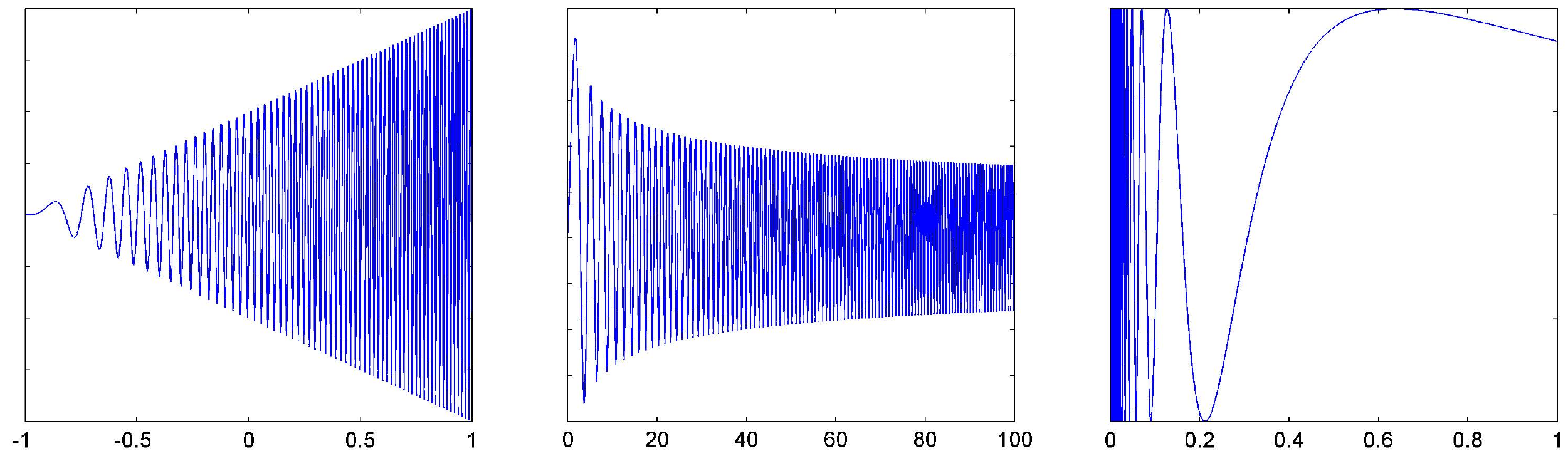}\\
  \caption{Highly oscillated functions. The left panel is $f_1(x)=(x+1)\sin(100(x+1)^2)$. The middle panel is Airy function: $f_2(x)=x^{-\frac{1}{4}}\sin(\frac{2}{3}x^{\frac{3}{2}})$. The right panel is Chirp function $f_3(x)=\sin\frac{x}{4}\cos(x^2)$.  }
  \label{fig:functions}
\end{figure}

We consider highly oscillating functions that can be approximated well by a low-rank TT format~\cite{khoromskij2015tensor}, as shown in Fig.~\ref{fig:functions}. We firstly tensorize the functional vector resulting in a $d$th-order tensor of size $n_1\times n_2\times\cdots\times n_d $, where isometric size is usually preferred, i.e., $n_1=n_2=\cdots=n_d = n$, with the total number of elements denoted by $N=n^d$.   The error bound (tolerance), denoted by $\epsilon_p=10^{-3}$, is given as the stopping criterion for all compared algorithms.   As shown in Table~\ref{Tab:Simulation1}, TR-SVD and TR-BALS  can obtain comparable results with TT-SVD  while outperform TT-SVD when noise is involved. These results indicate that TR representation is more robust to  noise than TT representation. 

\begin{table*}[t]
\centering
\caption{The functional data $f_1(x),f_2(x),f_3(x)$ is tensorized to 10th-order tensor ($4\times 4\times\ldots\times 4$). In the table, $\epsilon$, $\bar{r}$, $N_p$ denote relative error, average rank, and the total number of parameters, respectively.  }
\label{Tab:Simulation1}
\vspace*{2mm}
\begin{tabular}{cccccc c cccc}
\hline
 \multirow{2}{*}{}  &  \multicolumn{4}{c}{$f_1(x)$} &&  \multicolumn{4}{c}{$f_2(x)$ }  \\
\cline{2-5} \cline{7-10}
& $\epsilon$ & $\bar{r}$ & $N_p$ &  Time (s)  & & $\epsilon$ & $\bar{r}$ & $N_p$ &  Time (s) \\
\hline
 TT-SVD & 3e-4 & 4.4 & 1032 & 0.17 &&  3e-4 & 5 & 1360 & 0.16 \\
 TR-SVD &  3e-4 & 4.4 & 1032 & 0.17  &&  3e-4 & 5 & 1360 & 0.28\\
   TR-BALS & 9e-4 & 4.3 & 1052 & 4.6 &&  8e-4 & 4.9 & 1324 & 5.7\\
\hline
\hline
 \multirow{2}{*}{}  &  \multicolumn{4}{c}{$f_3(x)$ }  &&  \multicolumn{4}{c}{$f_1(x)+\mathcal{N}(0,\sigma), SNR=60dB$ }\\
\cline{2-5} \cline{7-10}
& $\epsilon$ & $\bar{r}$ & $N_p$ &  Time (s)  & & $\epsilon$ & $\bar{r}$ & $N_p$ &  Time (s)\\
\hline
 TT-SVD &  3e-4 & 3.7 & 680 & 0.16  &&  1e-3 & 16.6 & 13064 & 0.5\\
 TR-SVD &  5e-4 & 3.6 & 668 & 0.15  &&  1e-3 & 9.7 & 4644 & 0.4\\
 TR-BALS &  5e-4 & 3.7 & 728 & 3.4  &&  1e-3 & 4.2 & 1000 & 6.1\\
\hline
\end{tabular}
\end{table*}

\begin{table*}[t]
\caption{The results under different shifts of dimensions on functional data $f_2(x)$ with error bound set at $10^{-3}$. For the 10th-order tensor, all 9 dimension shifts were considered to compare the average rank $\bar{r}$. }
\label{tab:Simulation2}
\centering
\vspace*{2mm}
\begin{tabular}{c ccccccccc }
\hline
\multirow{2}{*}{} & \multicolumn{9}{c}{$\bar{r}$}  \\
 \cline{2-10}
 & 1 & 2 & 3 &  4 & 5 & 6 & 7 & 8 & 9 \\
 \hline
 TT-SVD & 5.2 & 5.8 & 6 &  6.2 & 7 & 7 & 8.5 & 14.6 & 8.4 \\
  TR-SVD & 5.2 & 5.8 & 5.9 &  6.2 & 9.6 & 10 & 14 & 12.7 & 6.5 \\
     TR-BALS & 5 & 4.9 & 5 &  4.9 & 4.9 & 5 & 5 & 4.8 & 4.9 \\
\hline
\end{tabular}
\end{table*}

It should be noted that TT representation has the property that $r_1=r_{d+1}=1$ and $r_k, k=2,\ldots,d-1$ are bounded by the rank of $k$-unfolding matrix of $\mat T_{\langle k\rangle}$, which limits its generalization ability and consistency when the tensor modes have been shifted or permuted.  To demonstrate this, we consider shifting the dimensions of $\tensor T$ of size ${n_1\times \cdots\times n_d}$ by $k$ times leading to $\overleftarrow{\tensor{T}}^k$ of size $n_{k+1}\times \cdots\times n_d\times n_1\times\cdots\times n_{k}$. As shown in Table~\ref{tab:Simulation2}, the average TT-ranks are varied dramatically along with the different shifts. In particular, when $k=8$, $\bar{r}_{tt}$ becomes 14.6, resulting in a large number of parameters $N_p=10376$.  In contrast to TT, TR-BALS can obtain consistent and compact representation.

\subsection{COIL-100 dataset}

\begin{table}[t]
\caption{The comparisons of different algorithms on Coil-100 dataset. $\epsilon$, $r_{max}$,$\bar{r}$ denote relative error, the maximum rank and the average rank, respectively.  }
\label{tab:coil100}
\centering
\begin{tabular}{c c c c c c c}
\hline
  & $\epsilon$ & $r_{max}$ & $\bar{r}$ & \pbox{1in} {Acc. (\%) \\($\rho=50 \%$)} & \pbox{1in} {Acc. (\%)\\ ($\rho=10 \%$)} \\
\hline
\multirow{5}{*} {TT-SVD}
 & 0.19 & 67  & 47.3 & 99.05  & 89.11  \\
 & 0.28 & 23  & 16.3 & 98.99  & 88.45  \\
 & 0.37 & 8   & 6.3  & 96.29  & 86.02  \\
 & 0.46 & 3   & 2.7  & 47.78  & 44.00  \\
 \hline
\multirow{5}{*} {TR-SVD}
 & 0.19 & 23  & 12.0 & 99.14  & 89.29  \\
 & 0.28 & 10  & 6.0  & 99.19 & 89.89  \\
 & 0.36 & 5   & 3.5  & 98.51  & 88.10  \\
 & 0.43 & 3   & 2.3  & 83.43  & 73.20 \\
 \hline
\end{tabular}
\end{table}

The Columbia Object Image Libraries (COIL)-100 dataset~\cite{nayar1996columbia}  contains 7200 color images of 100 objects (72 images per object) with different reflectance and complex geometric characteristics. Each image can be represented by a 3rd-order tensor of size $128\times 128\times 3$ and then is downsampled to $32\times 32\times 3$. Hence, the dataset can be finally organized as a 4th-order tensor of size $32\times 32\times 3\times 7200$. The number of features is determined by $r_4\times r_1$, while the flexibility of subspace bases is determined by $r_2, r_3$.  Subsequently, we apply the K-nearest neighbor (KNN) classifier with K=1 for classification. For detailed comparisons, we randomly select a certain ratio $\rho=50\%$ or $\rho=10\%$ samples as the training set and the rest as the test set. The classification performance is averaged over 10 times of random splitting. In Table \ref{tab:coil100},  $r_{max}$ of TR decompositions is much smaller than that of TT-SVD. It should be noted that TR representation, as compared to TT, can obtain more compact and discriminant representations.  

We have  conducted an additional experiment on video classifications (see detailed results in the Appendix).

\section{Conclusion}
\label{sec:conclusion}
We have proposed a novel tensor decomposition model, which provides an efficient representation for a very high-order tensor by a sequence of low-dimensional cores. The number of parameters in our model scales only linearly to the tensor order. To optimize the latent cores, we have presented two different algorithms: TR-SVD is a non-recursive algorithm that is stable and efficient, while TR-BALS can learn a more compact representation with adaptive TR-ranks. Furthermore, we have investigated  the properties on how the basic multilinear algebra can be performed efficiently by  operations over TR representations (i.e., cores), which provides a  powerful framework for processing large-scale data.  The experimental results verified the effectiveness of our proposed   algorithms.

\bibliographystyle{IEEEtran}
\bibliography{IEEEabrv,TensorNetwork}

\newpage
\appendix

\section{Proof of Theorem~\ref{theorem:invariance}}\label{sec:theorem:invariance}
\begin{proof}
It is obvious that (\ref{eq:TRD1}) can be rewritten as
\begin{multline}
T(i_1, i_2,\ldots, i_d) = \text{Tr}(\mat Z_2(i_2), \mat Z_3(i_3),\ldots, \mat Z_d(i_d), \mat Z_1(i_1))\\
 =\cdots =  \text{Tr}(\mat Z_d(i_d), \mat Z_1(i_1),\ldots, \mat Z_{d-1}(i_{d-1})).
\nonumber
\end{multline}
Therefore, we have
${\overleftarrow {\tensor {T}}^k} =\Re(\tensor Z_{k+1}, \ldots, \tensor Z_d, \tensor Z_{1},\ldots, \tensor Z_k)$.
\end{proof}

\section{Proof of Theorem~\ref{theorem:T1+T2}}\label{sec:theorem:T1+T2}
\begin{proof}
According to the definition of  TR decomposition, and the cores shown in (\ref{eq:TRsum1}), the $(i_1,\ldots,i_d)$th element of tensor $\tensor T_3$ can be written as
\begin{equation}
\begin{split}
   T_3(i_1,\ldots,i_d)       =&               \text{Tr}\left(
                  \begin{array}{cc}
                     \prod_{k=1}^d \mat Z_k(i_k) & 0 \\
                    0 &  \prod_{k=1}^d \mat Y_k(i_k) \\
                  \end{array}
                \right)=\text{Tr}\left(\prod_{k=1}^d \mat Z_k(i_k)\right) + \text{Tr}\left(\prod_{k=1}^d \mat Y_k(i_k)\right).
                \end{split}
\nonumber
\end{equation}
Hence, the \emph{addition} of tensors in the TR format can be performed by  merging of their cores.
\end{proof}

\section{Proof of Theorem~\ref{theorem:TRtimesvectors}}\label{sec:theorem:TRtimesvectors}
\begin{proof}
The \emph{multilinear product} between a tensor and vectors can be expressed by
\begin{equation}
\nonumber
\begin{split}
c=& \sum_{i_1,\ldots,i_d} T(i_1,\ldots, i_d) u_1(i_1)\cdots u_d(i_d)
= \sum_{i_1,\ldots,i_d} \text{Tr}\left(\prod_{k=1}^d \mat Z_k(i_k)\right) u_1(i_1)\cdots u_d(i_d)\\
=& \text{Tr}\left( \prod_{k=1}^d \left(\sum_{i_k=1}^{n_k} \mat Z_k(i_k)u_k(i_k)  \right)       \right).
\end{split}
\end{equation}
Thus, it can be written as a TR decomposition shown in (\ref{eq:TRmultiproduct}) where each core $\mat X_k \in\mathbb{R}^{r_k\times r_{k+1}}$ becomes a matrix. The computational complexity is equal to $\mathcal{O}(dnr^2)$.
\end{proof}

\section{Proof of Theorem~\ref{theorem:T1HadamardT2}}\label{sec:theorem:T1HadamardT2}
\begin{proof}
Each element in tensor $\tensor T_3$ can be written as
\begin{equation}
\nonumber
\begin{split}
T_3(i_1,\ldots,i_d) 
=&\text{Tr}\left(\prod_{k=1}^d \mat Z_k(i_k)\right)\text{Tr}\left(\prod_{k=1}^d \mat Y_k(i_k)\right)
= \text{Tr}\left\{\left(\prod_{k=1}^d \mat Z_k(i_k)\right)\otimes \left(\prod_{k=1}^d \mat Y_k(i_k)\right)\right\}\\
=& \text{Tr} \left\{\prod_{k=1}^d \Big( \mat Z_k(i_k) \otimes \mat Y_k(i_k) \Big) \right\}.
\end{split}
\end{equation}
Hence, $\tensor T_3$ can be also represented as TR format with its cores computed by (\ref{eq:TRhadama}), which costs $\mathcal{O}(dnq^2)$.
\end{proof}


\section{KTH video dataset}

\begin{figure}[htbp]
  \centering
  \includegraphics[width=1\columnwidth]{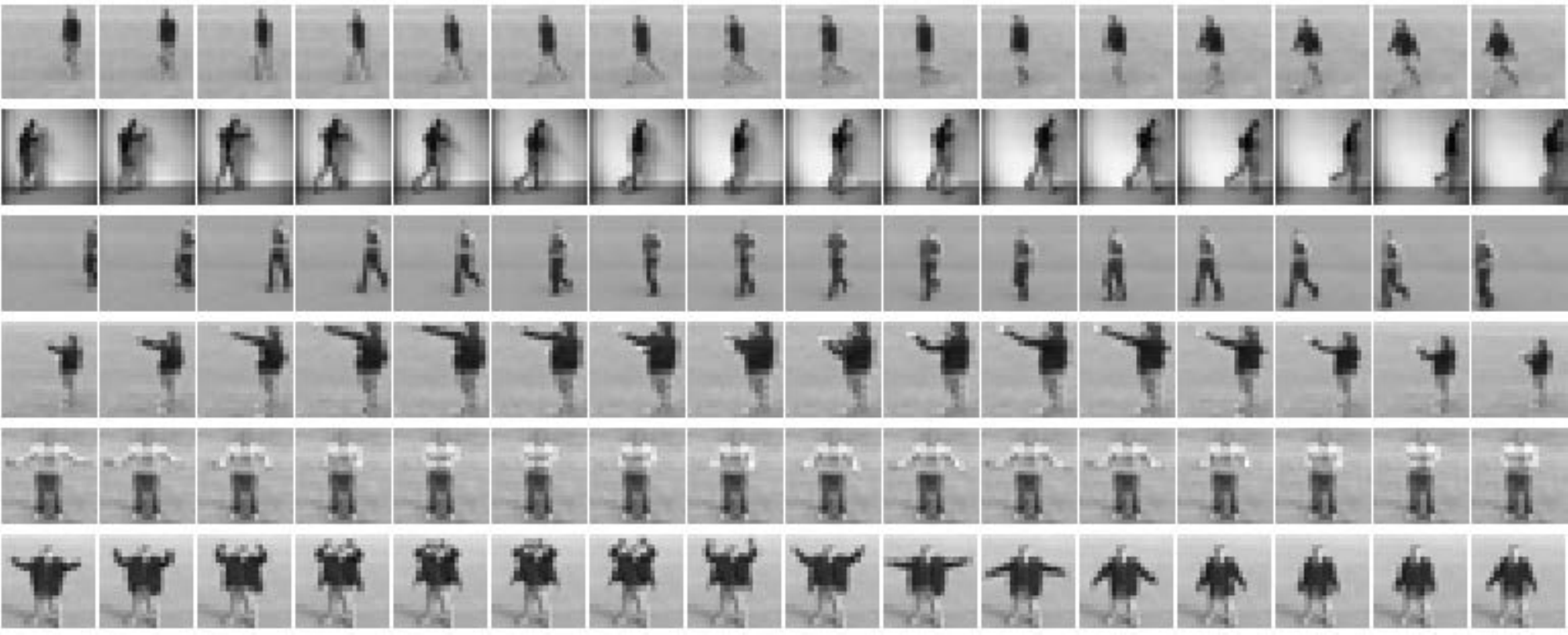}\\
  \caption{Video dataset consists of six types of human actions performed by 25 subjects in four different scenarios. From the top to bottom, six video examples corresponding to each type of actions are shown.  }
  \label{fig:KTHdataset}
\end{figure}


\begin{table}[htbp]
\renewcommand{\arraystretch}{1.1}
\caption{The comparisons of different algorithms on KTH dataset. $\epsilon$ denotes the obtained relative error; $r_{max}$ denotes maximum rank; $\bar{r}$ denotes the average rank;  and Acc. is the classification accuracy.  }
\label{tab:KTH}
\centering
\begin{tabular}{p{15mm} c c c c c}
\hline
 & $\epsilon$ & $r_{max}$ & $\bar{r}$ & Acc. ($5\times 5$-fold) \\
\hline
\multirow{3}{*}{ CP-ALS}
& 0.20 & 300 & 300  & 80.8 \% \\
& 0.30 & 40 & 40  & 79.3 \%\\
& 0.40 & 10 & 10  & 66.8 \%\\
\hline
\multirow{4}{*} {TT-SVD}
 & 0.20 & 139  & 78.0   & 84.8 \% \\
 & 0.29 & 38  & 27.3   & 83.5 \% \\
 & 0.38 & 14   & 9.3    & 67.8 \% \\
 \hline
\multirow{3}{*} {TR-SVD}
 & 0.20 & 99  & 34.2   & 78.8 \% \\
 & 0.29 & 27  & 12.0    & 87.7 \% \\
 & 0.37 & 10   & 5.8   & 72.4 \% \\
\hline
\end{tabular}
\end{table}

We test the TR representation for KTH video database~\cite{laptev2006local} containing six types of human actions (walking, jogging, running, boxing, hand waving and hand clapping) performed several times by 25 subjects in four different scenarios: outdoors, outdoors with scale variation, outdoors with different clothes and indoors as illustrated in Fig.~\ref{fig:KTHdataset}. There are 600 video sequences for each combination of 25 subjects, 6 actions and 4 scenarios. Each video sequence was downsampled to $20\times 20\times 32$. Finally, we can organize the dataset as a tensor of size $20\times 20\times 32\times 600$.  For extensive comparisons, we choose different error bound $\epsilon_p\in \{0.2,0.3,0.4\}$.  In Table \ref{tab:KTH}, we can see that TR representations achieve better compression ratio reflected by smaller $r_{max}, \bar{r}$ than that of TT-SVD, while TT-SVD achieves better compression ratio than CP-ALS. For instance, when $\epsilon\approx 0.2$, CP-ALS requires $r_{max}=300$, $\bar{r}=300$; TT-SVD requires $r_{max}=139$, $\bar{r}=78$, while TR-SVD only requires $r_{max}=99$, $\bar{r}=34.2$. For classification performance, we observe that the best accuracy ($5\times 5$-fold cross validation) achieved by CP-ALS, TT-SVD, TR-SVD are 80.8\%, 84.8\%, 87.7\%, respectively. Note that these classification performances might not  be the state-of-the-art on this dataset, we mainly focus on the comparisons of representation ability among CP, TT, and TR decomposition frameworks. To obtain the best performance, we may apply the powerful  feature extraction methods to TT or TR representations of dataset. It should be noted that TR decompositions achieve the best classification accuracy when $\epsilon= 0.29$, while TT-SVD and CP-ALS achieve their best classification accuracy when $\epsilon=0.2$. This indicates that TR decomposition can preserve more discriminant information even when the approximation error is relatively high.   This experiment demonstrates that TR decompositions are effective for unsupervised feature representation due to their flexibility of TR-ranks and high compression ability.

\end{document}

%% file: DefinedComand.tex
\newcommand{\tensor}[1]{\boldsymbol{\mathcal{#1}}}

\newcommand{\mat}[1]{\mathbf{#1}}
\newcommand{\vect}[1]{\mathbf{#1}}

\theoremstyle{plain}

%% file: nips_2017_Sugiyama2.bbl
\begin{thebibliography}{10}
\providecommand{\url}[1]{#1}
\csname url@samestyle\endcsname
\providecommand{\newblock}{\relax}
\providecommand{\bibinfo}[2]{#2}
\providecommand{\BIBentrySTDinterwordspacing}{\spaceskip=0pt\relax}
\providecommand{\BIBentryALTinterwordstretchfactor}{4}
\providecommand{\BIBentryALTinterwordspacing}{\spaceskip=\fontdimen2\font plus
\BIBentryALTinterwordstretchfactor\fontdimen3\font minus
  \fontdimen4\font\relax}
\providecommand{\BIBforeignlanguage}[2]{{%
\expandafter\ifx\csname l@#1\endcsname\relax
\typeout{** WARNING: IEEEtran.bst: No hyphenation pattern has been}%
\typeout{** loaded for the language `#1'. Using the pattern for}%
\typeout{** the default language instead.}%
\else
\language=\csname l@#1\endcsname
\fi
#2}}
\providecommand{\BIBdecl}{\relax}
\BIBdecl

\bibitem{anandkumar2014tensor}
A.~Anandkumar, R.~Ge, D.~J. Hsu, S.~M. Kakade, and M.~Telgarsky, ``Tensor
  decompositions for learning latent variable models.'' \emph{Journal of
  Machine Learning Research}, vol.~15, no.~1, pp. 2773--2832, 2014.

\bibitem{romera2013multilinear}
B.~Romera-Paredes, H.~Aung, N.~Bianchi-Berthouze, and M.~Pontil, ``Multilinear
  multitask learning,'' in \emph{International Conference on Machine Learning},
  2013, pp. 1444--1452.

\bibitem{kanagawa2016gaussian}
H.~Kanagawa, T.~Suzuki, H.~Kobayashi, N.~Shimizu, and Y.~Tagami, ``Gaussian
  process nonparametric tensor estimator and its minimax optimality,'' in
  \emph{International Conference on Machine Learning (ICML2016)}, 2016, pp.
  1632--1641.

\bibitem{cong2015tensor}
F.~Cong, Q.-H. Lin, L.-D. Kuang, X.-F. Gong, P.~Astikainen, and T.~Ristaniemi,
  ``Tensor decomposition of {EEG} signals: a brief review,'' \emph{Journal of
  neuroscience methods}, vol. 248, pp. 59--69, 2015.

\bibitem{Zhou-PIEEE}
G.~Zhou, Q.~Zhao, Y.~Zhang, T.~Adali, S.~Xie, and A.~Cichocki, ``Linked
  component analysis from matrices to high-order tensors: Applications to
  biomedical data,'' \emph{Proceedings of the IEEE}, vol. 104, no.~2, pp.
  310--331, 2016.

\bibitem{bro1997parafac}
R.~Bro, ``{PARAFAC. Tutorial and applications},'' \emph{Chemometrics and
  intelligent laboratory systems}, vol.~38, no.~2, pp. 149--171, 1997.

\bibitem{CPD-Comon15}
J.~Goulart, M.~Boizard, R.~Boyer, G.~Favier, and P.~Comon, ``Tensor cp
  decomposition with structured factor matrices: Algorithms and performance,''
  \emph{IEEE Journal of Selected Topics in Signal Processing}, 2015.

\bibitem{zhao2015bayesian}
Q.~Zhao, L.~Zhang, and A.~Cichocki, ``Bayesian {CP} factorization of incomplete
  tensors with automatic rank determination,'' \emph{IEEE Transactions on
  Pattern Analysis and Machine Intelligence}, vol.~37, no.~9, pp. 1751--1763,
  2015.

\bibitem{tucker1966some}
L.~R. Tucker, ``Some mathematical notes on three-mode factor analysis,''
  \emph{Psychometrika}, vol.~31, no.~3, pp. 279--311, 1966.

\bibitem{HOOI:Lathauwer:2000}
L.~{De Lathauwer}, B.~{De Moor}, and J.~Vandewalle, ``On the best rank-1 and
  rank-({R1,R2,. . .,RN}) approximation of higher-order tensors,'' \emph{SIAM
  J. Matrix Anal. Appl.}, vol.~21, pp. 1324--1342, 2000.

\bibitem{qiinfinite}
Z.~Xu, F.~Yan, and A.~Qi, ``Infinite {Tucker} decomposition: Nonparametric
  {Bayesian} models for multiway data analysis,'' in \emph{Proceedings of the
  29th International Conference on Machine Learning (ICML-12)}, 2012, pp.
  1023--1030.

\bibitem{wu2014multifactor}
Q.~Wu, L.~Zhang, and A.~Cichocki, ``Multifactor sparse feature extraction using
  convolutive nonnegative tucker decomposition,'' \emph{Neurocomputing}, vol.
  129, pp. 17--24, 2014.

\bibitem{cichocki2016tensor}
A.~Cichocki, N.~Lee, I.~Oseledets, A.-H. Phan, Q.~Zhao, D.~P. Mandic
  \emph{et~al.}, ``Tensor networks for dimensionality reduction and large-scale
  optimization: Part 1 low-rank tensor decompositions,'' \emph{Foundations and
  Trends{\textregistered} in Machine Learning}, vol.~9, no. 4-5, pp. 249--429,
  2016.

\bibitem{oseledets2011tensor}
I.~V. Oseledets, ``Tensor-train decomposition,'' \emph{SIAM Journal on
  Scientific Computing}, vol.~33, no.~5, pp. 2295--2317, 2011.

\bibitem{novikov2015tensorizing}
A.~Novikov, D.~Podoprikhin, A.~Osokin, and D.~P. Vetrov, ``Tensorizing neural
  networks,'' in \emph{Advances in Neural Information Processing Systems},
  2015, pp. 442--450.

\bibitem{NIPS2016_6211}
\BIBentryALTinterwordspacing
E.~Stoudenmire and D.~J. Schwab, ``Supervised learning with tensor networks,''
  in \emph{Advances in Neural Information Processing Systems 29}, D.~D. Lee,
  M.~Sugiyama, U.~V. Luxburg, I.~Guyon, and R.~Garnett, Eds.\hskip 1em plus
  0.5em minus 0.4em\relax Curran Associates, Inc., 2016, pp. 4799--4807.
  [Online]. Available:
  \url{http://papers.nips.cc/paper/6211-supervised-learning-with-tensor-networks.pdf}
\BIBentrySTDinterwordspacing

\bibitem{7207289}
J.~A. Bengua, H.~N. Phien, and H.~D. Tuan, ``Optimal feature extraction and
  classification of tensors via matrix product state decomposition,'' in
  \emph{2015 IEEE International Congress on Big Data}, June 2015, pp. 669--672.

\bibitem{oseledets2010tt}
I.~Oseledets and E.~Tyrtyshnikov, ``{TT}-cross approximation for
  multidimensional arrays,'' \emph{Linear Algebra and its Applications}, vol.
  432, no.~1, pp. 70--88, 2010.

\bibitem{Kolda09}
T.~Kolda and B.~Bader, ``Tensor decompositions and applications,'' \emph{SIAM
  Review}, vol.~51, no.~3, pp. 455--500, 2009.

\bibitem{Holtz-TT-2012}
S.~Holtz, T.~Rohwedder, and R.~Schneider, ``The alternating linear scheme for
  tensor optimization in the tensor train format,'' \emph{SIAM J. Scientific
  Computing}, vol.~34, no.~2, 2012.

\bibitem{khoromskij2015tensor}
B.~N. Khoromskij, ``Tensor numerical methods for multidimensional {PDEs}:
  theoretical analysis and initial applications,'' \emph{ESAIM: Proceedings and
  Surveys}, vol.~48, pp. 1--28, 2015.

\bibitem{nayar1996columbia}
S.~Nayar, S.~Nene, and H.~Murase, ``Columbia object image library (coil 100),''
  \emph{Department of Comp. Science, Columbia University, Tech. Rep.
  CUCS-006-96}, 1996.

\bibitem{laptev2006local}
I.~Laptev and T.~Lindeberg, ``Local descriptors for spatio-temporal
  recognition,'' in \emph{Spatial Coherence for Visual Motion Analysis}.\hskip
  1em plus 0.5em minus 0.4em\relax Springer, 2006, pp. 91--103.

\end{thebibliography}
